\documentclass[12pt]{amsart}

\newtheorem{theorem}{Theorem}[section]
\newtheorem{corollary}[theorem]{Corollary}
\newtheorem{lemma}[theorem]{Lemma}
\newtheorem{proposition}[theorem]{Proposition}

\theoremstyle{theorem}
\newtheorem{innercustomthm}{Theorem}
\newenvironment{customthm}[1]
  {\renewcommand\theinnercustomthm{#1}\innercustomthm}
  {\endinnercustomthm}

\theoremstyle{definition}

\theoremstyle{remark}

\numberwithin{equation}{section}

\def\no{\nonumber}

\newcommand{\der}[1]{\frac{\partial}{\partial #1}}
\newcommand{\fder}[2]{\frac{\partial #1}{\partial #2}}

\begin{document}
\title[Shape optimization for the Steklov problem in higher dimensions]
{Shape optimization for the Steklov problem in higher dimensions}
\author{Ailana Fraser}
\address{Department of Mathematics \\
                 University of British Columbia \\
                 Vancouver, BC V6T 1Z2}
\email{afraser@math.ubc.ca}
\author{Richard Schoen}
\address{Department of Mathematics \\
                 University of California \\
                 Irvine, CA 92617}
\email{rschoen@math.uci.edu}
\thanks{
A. Fraser was partially supported by  the 
Natural Sciences and Engineering Research Council of Canada and R. Schoen
was partially supported by NSF grant DMS-1710565. Part of this work was done while the authors were visiting ETH, and they would like to thank the FIM for its support.}

\begin{abstract} 
We show that the ball does not maximize the first nonzero Steklov eigenvalue among all contractible domains of fixed boundary volume in $\mathbb{R}^n$ when $n \geq 3$. 
This is in contrast to the situation when $n=2$, where a result of Weinstock from 1954 shows that the disk uniquely maximizes the first Steklov eigenvalue among all simply connected domains in the plane having the same boundary length.
When $n \geq 3$, we show that increasing the number of boundary components does not
increase the normalized (by boundary volume) first Steklov eigenvalue. This is in contrast to recent results which have been obtained for surfaces.
\end{abstract}

\maketitle

\section{Introduction}
In this paper we study optimization problems for Steklov eigenvalues on manifolds $M^n$
with boundary. The main theme of the paper is to show that some of the refined
results which are true for surfaces ($n=2$) do not hold in higher dimensions ($n\geq 3$).

We first consider the question of optimizing the first nonzero Steklov eigenvalue $\sigma_1(\Omega)$
for suitably normalized domains $\Omega$ in $\mathbb R^n$. A theorem of F. Brock \cite{Br}
says that a round ball uniquely maximizes $\sigma_1$ over all smooth domains with the same
(or larger) volume. On the other hand for $n=2$ there is a stronger result of R. Weinstock \cite{W} which says that the unit disk in the plane uniquely maximizes $\sigma_1$ over all simply connected
domains with the same (or larger) boundary length. From the isoperimetric inequality, any domain
which has the same volume as a ball necessarily has boundary volume which is at least as
large. Thus we see that for simply connected plane domains Weinstock's theorem implies
Brock's theorem. On the other hand Brock's theorem holds for arbitrary plane domains and
domains in $\mathbb{R}^n$ for $n\geq 3$. This leads to the question of whether there is
an analogue to Weinstock's theorem in higher dimensions. The question of whether the
ball in higher dimensions maximizes $\sigma_1$ over domains which are diffeomorphic
to the ball (contractible domains) has been open. Some related questions were posed in 
\cite[page 4]{PS} and  \cite[Open Problem 2]{GP4}. 
In this paper we show that this is not true in higher dimensions.
\begin{theorem} \label{theorem:ball}
For $n\geq 3$ there is a smooth contractible domain $\Omega$
with $|\partial\Omega|=|\partial \mathbb{B}_1|$ where $\mathbb{B}_1$ is a unit ball in
$\mathbb R^n$, but with $\sigma_1(\Omega)>\sigma_1(\mathbb{B}_1)=1$.
\end{theorem}

In Proposition \ref{proposition:bound} we also give an explicit upper bound on $\sigma_1(\Omega)$
for any smooth domain in $\mathbb R^n$ in terms of its boundary volume. This leaves open
the question of finding the sharp value for this upper bound. Theorem \ref{theorem:ball} shows that
it is strictly larger than its value for a ball. 

In order to prove Theorem \ref{theorem:ball} we first consider the annular domain $\Omega_\epsilon=\mathbb{B}_1\setminus \mathbb{B}_\epsilon$. We show in  Proposition \ref{proposition:punctured ball} that the $k$th Steklov eigenvalue is decreased by approximately a positive constant times $\epsilon^{2k+n-2}$. 
When $k=1$ the exponent is equal to $n$, and it follows that when $\epsilon$ is small the normalized first
Steklov eigenvalue $\sigma_1(\Omega_\epsilon)|\partial\Omega_\epsilon|^{\frac{1}{n-1}}$
is strictly larger than that of $\mathbb{B}_1$ (actually the same is true for higher eigenvalues). 
For $n\geq 3$, we then show that we can modify the domain $\Omega_\epsilon$ to make it contractible while changing the normalized first Steklov eigenvalue by an arbitrarily small amount. This is accomplished by
adding a small tube joining the boundary components and showing that the construction can be done keeping the normalized eigenvalue nearly unchanged. This construction leads to a more general question about boundary connectedness.

Another result for $n=2$ which was discovered in \cite{FS2} is that by adding an
extra boundary component to a surface the normalized first Steklov eigenvalue
$\sigma_1L$ (where $L$ is the boundary length) can be made strictly larger. This was
used to show that surfaces of genus $0$ (homeomorphic to plane domains) which
maximize $\sigma_1$ for their boundary length must have an infinite number of
boundary components. The question of whether a similar phenomenon might be true
in higher dimensions was posed in  \cite[following Open Problem 2]{GP4}.
We show here that this is also not true for manifolds with $n\geq 3$. Specifically
we show that the number of boundary components does not effect the maximum
value of the normalized first Steklov eigenvalue.
\begin{theorem} \label{theorem:connected-boundary}
Given any compact Riemannian manifold $\Omega^n$ with non-empty
boundary and $n\geq 3$, and given any $\epsilon>0$ there exists a smooth subdomain $\Omega_\epsilon$ of $\Omega$ with connected boundary such that 
\[
|\Omega|-|\Omega_\epsilon|<\epsilon,\ ||\partial\Omega|-|\partial\Omega_\epsilon||<\epsilon,
\ \mbox{and}\ |\sigma_1(\Omega)-\sigma_1(\Omega_\epsilon)|<\epsilon.
\]
\end{theorem}
In Section \ref{section:bound} of the paper we give an explicit coarse upper bound on the normalized first
Steklov eigenvalue of a domain in $\mathbb R^n$. This is done by using stereographic
projection and a balancing argument. This is a less general but more precise bound
than that of \cite{CEG} (see also \cite{H}).

In Section \ref{section:annulus} we do the asymptotic calculation of the $k$th Steklov eigenvalue of
$\mathbb{B}_1\setminus \mathbb{B}_\epsilon$. This calculation had been done
previously by the authors for $k=1$ and $n=2$ (cf. \cite[Proposition 4.2]{FS2})
and \cite[Example 4.2.5]{GP4}, 
and  was done for $k=1$ and $n\geq 2$ by E. Martel (see \cite[Remark 4.2.8]{GP4}).

In Section \ref{section:surgery} we prove the main results concerning the effect of boundary connectedness.
This involves delicate estimation of the first Steklov eigenvalue for domains with small
tubes connecting boundary components.

\section{Upper bounds} \label{section:bound}

In dimension $n=2$, for any compact Riemannian surface, the $k$-th normalized Steklov eigenvalue is bounded above in terms of $k$, the genus $\gamma$ and the number of boundary components $b > 0$ of the surface, in the most general form by Karpukhin \cite{Ka},
\[
     \sigma_k (\Sigma) |\partial \Sigma| \leq 2\pi (k+\gamma+b-1)
\]
(see \cite{GP3}, \cite{FS1}, \cite{HPS}, \cite{W} for earlier results) and also in terms of only the genus and $k$,
\[
    \sigma_k (\Sigma) |\partial \Sigma| \leq Ak + B \gamma
\]
for constants $A$ and $B$ (see \cite{H}, \cite{CEG}, \cite{Ko}). For simply-connected domains in $\mathbb{R}^2$, the first bound is sharp, and the $k$-th normalized Steklov eigenvalue is maximized in the limit by a disjoint union of $k$ identical disks for any $k \geq 1$ (\cite{GP1}, \cite{W}). In general the bounds are not sharp, however sharp bounds are known for the first nonzero normalized Steklov eigenvalue for the annulus and the M\"obius band (\cite{FS2}), and the authors proved existence of a metric that maximizes the first nonzero normalized eigenvalue on any compact orientable surface of genus zero (\cite{FS2}). Moreover, it was shown in \cite{FS2} that the maximum value of $\sigma_1(\Sigma)|\partial \Sigma|$ over all smooth metrics on a compact orientable surface $\Sigma$ of genus zero, is strictly increasing in the number $b$ of boundary components, and converges to $4\pi$ as $b$ tends to infinity. Thus, the asymptotically sharp upper bound for surfaces of genus zero is $4\pi$. 

In higher dimensions, it was shown in \cite{CEG} (see also a generalization in \cite{H}) that if $M$ is Riemannian manifold of dimension $n \geq 2$ that is conformally equivalent to a complete Riemannian manifold  with non-negative Ricci curvature, then for any domain $\Omega \subset M$,
\[
      \sigma_k(\Omega) |\partial \Omega|^{\frac{1}{n-1}}
      \leq \frac{\alpha(n)}{{I(\Omega)^{\frac{n-2}{n-1}}}} \, k^{\frac{2}{n}}
\]
where $I(\Omega)=|\partial \Omega| / |\Omega|^{\frac{n-1}{n}}$ is the isoperimetric ratio. 
In particular, for any bounded domain $\Omega$ in $\mathbb{R}^n$, by the classical isoperimetric inequality, it follows that the normalized Steklov eigenvalues are uniformly bounded above, 
$\sigma_k(\Omega) |\partial \Omega|^{\frac{1}{n-1}} \leq C(n) k^{\frac{1}{n-1}}$. We observe that in the special case when $\Omega \subset \mathbb{R}^n$, with $n \geq 2$, an explicit bound can be directly obtained easily for $k=1$ as follows.

\begin{proposition} \label{proposition:bound}
If $\Omega$ is a domain in $\mathbb{R}^n$, $n \geq 2$, then
\[
      \sigma_1(\Omega) |\partial \Omega|^{\frac{1}{n-1}}     
      \leq  \frac{n^{\frac{1}{n-1}}|\mathbb{S}^n|^{\frac{2}{n}}}{|\mathbb{B}^n|^{\frac{n-2}{n(n-1)}}}.
\]
\end{proposition}

\begin{proof}
Using stereographic projection, and a standard balancing argument, there exists a conformal map $F:\Omega \rightarrow \mathbb{S}^n \subset \mathbb{R}^{n+1}$ with $\int_{\partial \Omega} F =0$. Using the component functions $F_i$, $i=1, \ldots , n+1$, as test functions in the variational characterization of $\sigma_1$, we have
\[
    \sigma_1 \int_{\partial \Omega} F_i^2 \leq \int_\Omega |\nabla F_i|^2.
\]
Summing on $i$, and applying H\"older's inequality,
\[
    \sigma_1 |\partial \Omega |  \leq \int_\Omega |\nabla F |^2  
     \leq \left( \int_\Omega |\nabla F|^n  \right)^{\frac{2}{n}} |\Omega|^{\frac{n-2}{n}}
    \leq n |\mathbb{S}^n|^{\frac{2}{n}} |\Omega|^{\frac{n-2}{n}}
    \leq  \frac{n^{\frac{1}{n-1}}|\mathbb{S}^n|^{\frac{2}{n}}}{|\mathbb{B}^n|^{\frac{n-2}{n(n-1)}}} |\partial \Omega|^{\frac{n-2}{n-1}},
\]
where the last inequality follows from the isoperimetric inequality 
$|\Omega|/|\mathbb{B}^n| \leq (|\partial \Omega|/|\mathbb{S}^{n-1}|)^{\frac{n}{n-1}}$ and the formula 
$|\mathbb{S}^{n-1}|=n|\mathbb{B}^n|$. Simplifying, we obtain the desired bound.
\end{proof}

\section{Dirichlet-to-Neumann spectrum of $\mathbb{B}^n_1 \setminus \mathbb{B}^n_\epsilon$} 
\label{section:annulus}

Throughout the paper we let $\mathbb{B}^n_\rho$ denote the ball of radius $\rho$ in $\mathbb{R}^n$, and use spherical coordinates $\rho$, $\phi_1, \ldots, \phi_{n-1}$ on $\mathbb{R}^n$.
In this section we calculate the Dirichlet-to-Neumann spectrum of $\mathbb{B}^n_1 \setminus \mathbb{B}^n_\epsilon$, where $0<\epsilon<1$ is small, and show that $k$-th nonzero normalized Steklov eigenvalue of $\mathbb{B}^n_1 \setminus \mathbb{B}^n_\epsilon$ is strictly greater than that of the ball $\mathbb{B}^n_1$. For $k=1$ this was verified by E. Martel \cite[Remark 4.2.8]{GP4}.
\begin{proposition} \label{proposition:punctured ball}
For $\epsilon$ sufficiently small, $0< \epsilon < 1$, the $k$-th \textit{\textbf{distinct}} Steklov eigenvalue of $\mathbb{B}^n_1 \setminus \mathbb{B}^n_\epsilon$ for  $n \geq 3$ is
\[
      \sigma_k = k-\frac{k(2k+n-2)}{k+n-2}\epsilon^{2k+n-2} + O(\epsilon^{2k+n-1}),
\]
for $k=1, \, 2, \, 3, \ldots$.
In particular, for $\epsilon$ sufficiently small the first nonzero Steklov eigenvalue is
\[
      \sigma_1 = 1-\frac{n}{n-1}\epsilon^{n} + O(\epsilon^{n+1}),
\]  
and 
\[
     \sigma_k(\mathbb{B}^n_1) |\partial \mathbb{B}^n_1|^{\frac{1}{n-1}}
     < \sigma_k(\mathbb{B}^n_1\setminus \mathbb{B}^n_\epsilon) |\partial(\mathbb{B}^n_1\setminus \mathbb{B}^n_\epsilon)|^{\frac{1}{n-1}}.
\]
\end{proposition}
\begin{proof}
The outward unit normal vector on $\partial \mathbb{B}^n_1$ is given by $\eta=\der{\rho}$ and on $\partial \mathbb{B}^n_\epsilon$ by $\eta=-\der{\rho}$. To compute the Dirichlet-to-Neumann spectrum we separate variables and look for harmonic functions of the form $u(\rho,\phi_1,\ldots,\phi_{n-1})=\alpha(\rho) \beta(\phi_1, \ldots, \phi_{n-1})$. By standard methods if $n > 2$ we obtain solutions for any nonnegative integer $k$ given by
\[
         \alpha(\rho)=a\rho^k+b\rho^{-k+2-n}.
\]
In order to be an eigenfunction of the Dirichlet-to-Neumann map we must have $u_\eta=\sigma u$ on $\partial(\mathbb{B}^n_1\setminus \mathbb{B}^n_\epsilon)$, or
$\alpha'(1)=\sigma \alpha(1)$ on $\partial \mathbb{B}^n_1$ and 
$\alpha'(\epsilon)=-\sigma \alpha(\epsilon)$ on $\partial \mathbb{B}^n_\epsilon$. For each nonnegative integer $k$ the conditions become
\begin{align*}
      a k + b (-k+2-n)&=\sigma (a+b) \\
      ak\epsilon^{k-1}+b(-k+2-n)\epsilon^{-k+1-n}
      &=-\sigma(a\epsilon^k+b\epsilon^{-k+2-n}).
\end{align*}
Factoring out $a$ and $b$ these become
\begin{align*}
      a(k-\sigma)&=b(\sigma+k-2+n) \\
      a(k\epsilon^{k-1}+\sigma\epsilon^k)
      &=b(-\sigma \epsilon^{-k+2-n}+(k-2+n)\epsilon^{-k+1-n}).
\end{align*}
Using the first equation to eliminate $a$ and dividing by $b$ (which must be nonzero) we get 
\[
         \frac{\sigma+k-2+n}{k-\sigma}(k\epsilon^{k-1}+\sigma\epsilon^k)
         =-\sigma\epsilon^{-k+2-n}+(k+n-2)\epsilon^{-k+1-n},
\]
which gives the quadratic equation for $\sigma$
\begin{align*}
     \sigma^2(\epsilon^k-\epsilon^{-k+2-n})
     &+\sigma(k\epsilon^{k-1}+(k-2+n)\epsilon^k+k\epsilon^{-k+2-n})
                         +(k+n-2)\epsilon^{-k+1-n}) \\
     &+(k+n-2)k(\epsilon^{k-1}-\epsilon^{-k+1-n})=0.
\end{align*}
Multiplying through by $\epsilon^{k+n-1}$ we may rewrite this as
\begin{align*}
     \sigma^2(\epsilon-\epsilon^{2k+n-1})
     & -\sigma((k+n-2)\epsilon^{2k+n-1}+k\epsilon^{2k+n-2}+k\epsilon)+k+n-2) \\
     & +(k+n-2)k(1-\epsilon^{2k+n-2})=0.
\end{align*}
Letting $A$, $B$, $C$ be the coefficients in this quadratic equation, $A\sigma^2+B\sigma+C=0$, we calculate that
\begin{align*}
    B&^2-4AC  \\ &=
    (k+n-2)^2 -2k(k+n-2)\epsilon +k^2 \epsilon^2 + 2k(k+n-2)\epsilon^{2k+n-2} 
                   + c(n,k) \epsilon^{2k+n-1} + O(\epsilon^{2k+n}) \\
          &= (k+n-2-k\epsilon)^2 + 2k(k+n-2)\epsilon^{2k+n-2} 
                + c(n,k) \epsilon^{2k+n-1} + O(\epsilon^{2k+n}) \\
          &= (k+n-2-k\epsilon)^2 \left[ 1+ \frac{2k(k+n-2)}{D^2}\epsilon^{2k+n-2} 
                + \frac{c(n,k)}{D^2} \epsilon^{2k+n-1} + O(\epsilon^{2k+n}) \right]
\end{align*}  
where $c(n,k)=2k^2+2(k+n-2)^2+8k(k+n-2)$ and $D=k+n-2-k\epsilon$. Using the expansion $\sqrt{1+x}=1+\frac{1}{2}x-\frac{1}{8}x^2 + \cdots$, we have
\begin{align*}
     \sqrt{B^2-4AC}
     &=D\left[ 1+\frac{1}{2}\left( \frac{2k(k+n-2)}{D^2}\epsilon^{2k+n-2} 
                + \frac{c(n,k)}{D^2} \epsilon^{2k+n-1}\right) + O(\epsilon^{2k+n}) \right] \\
     &=D+\frac{k(k+n-2)}{D}\epsilon^{2k+n-2} 
                + \frac{c(n,k)}{2D} \epsilon^{2k+n-1} + O(\epsilon^{2k+n}).
\end{align*}
Now since
\[
      \frac{1}{D}=\frac{1}{k+n-2-k\epsilon}
      =\frac{1}{k+n-2}\frac{1}{1-\frac{k}{k+n-2}\epsilon}
      =\frac{1}{k+n-2}\left[ 1 + \frac{k}{k+n-2} \epsilon + O(\epsilon^2) \right],
\]
we get that
\begin{align*}
   & \sqrt{B^2-4AC}  \\
   & \quad =k+n-2-k\epsilon + k \left[ 1+\frac{k}{k+n-2}\epsilon \right] \epsilon^{2k+n-2}
      +\frac{c(n,k)}{2(k+n-2)} \epsilon^{2k+n-1} +O(\epsilon^{2k+n}) \\
   & \quad = k+n-2-k\epsilon +k\epsilon^{2k+n-2}
     +\left[\frac{k^2}{k+n-2}+\frac{c(n,k)}{2(k+n-2)}\right] \epsilon^{2k+n-1} 
     +O(\epsilon^{2k+n}).
\end{align*}
Set 
\[
    c'(n,k):= \frac{k^2}{k+n-2}+\frac{c(n,k)}{2(k+n-2)}
    =\frac{2k^2+(k+n-2)^2+4k(k+n-2)}{k+n-2}.
\]
Also, $A=\epsilon(1-\epsilon^{2k+n-2})$, and
\[
     \frac{1}{A}=\frac{1}{\epsilon(1-\epsilon^{2k+n-2})}
     =\epsilon^{-1}(1+\epsilon^{2k+n-2}+O(\epsilon^{4k+2n-4})).
\]
Using this, we see that the quadratic equation for $\sigma$ has roots
\begin{align*}
    \sigma&=\frac{1}{2}(\epsilon^{-1}+\epsilon^{2k+n-3})
    \Big[\;(k +n-2)\epsilon^{2k+n-1}+k\epsilon^{2k+n-2}+k\epsilon+k+n-2  \\
    & \qquad \qquad \qquad \qquad \quad \pm \Big( k+n-2 -k\epsilon +k\epsilon^{2k+n-2}
                        +c'(n,k) \epsilon^{2k+n-1}+O(\epsilon^{2k+n}) \Big) \; \Big].
\end{align*}
Hence there are two positive roots $\sigma_k^{(1)}<\sigma_k^{(2)}$ given by
\begin{align*}
    \sigma_k^{(2)} & = O(\epsilon^{-1}) \\
    \sigma_k^{(1)} &=\frac{1}{2}(\epsilon^{-1}+\epsilon^{2k+n-2})
    \Big[\;2k\epsilon+(k+n-2-c'(n,k))\epsilon^{2k+n-1}+O(\epsilon^{2k+n}) \Big] \\
    &=k+\frac{1}{2}(k+n-2-c'(n,k))\epsilon^{2k+n-2}
         +k\epsilon^{2k+n-2}+O(\epsilon^{2k+n-1}) \\
    &=k-\frac{k(2k+n-2)}{k+n-2} \epsilon^{2k+n-2} +O(\epsilon^{2k+n-1}).
\end{align*}
For any given $k$ and $\epsilon$ sufficiently small, we see that $\sigma_k^{(2)}$ is the $k$-th {\bf distinct} eigenvalue of $\mathbb{B}^n_1 \setminus \mathbb{B}^n_\epsilon$, and the multiplicity of the the $k$-th eigenvalue of $\mathbb{B}^n_1 \setminus \mathbb{B}^n_\epsilon$ and of $\mathbb{B}^n_1$ must be the same.
For $\mathbb{B}^n_1$, it is well known that the $k$-th {\bf distinct} nonzero Steklov eigenvalue is $\sigma_k(\mathbb{B}^n_1)=k$. On the other hand, 
\begin{align*}
       |\partial (\mathbb{B}^n_1 \setminus \mathbb{B}^n_\epsilon)|^{\frac{1}{n-1}}
       &=|\partial \mathbb{B}^n_1|^{\frac{1}{n-1}}(1+\epsilon^{n-1})^{\frac{1}{n-1}} \\
       &=|\partial \mathbb{B}^n_1|^{\frac{1}{n-1}}\left(1+\frac{1}{n-1}\epsilon^{n-1}+O(\epsilon^{2n-2})\right).
\end{align*}
Therefore,
\begin{align*}
      \sigma_k(& \mathbb{B}^n_1  \setminus \mathbb{B}^n_\epsilon) \; 
      |\partial (\mathbb{B}^n_1 \setminus \mathbb{B}^n_\epsilon)|^{\frac{1}{n-1}} \\
      &= \left(k-\frac{k(2k+n-2)}{k+n-2} \epsilon^{2k+n-2} +O(\epsilon^{2k+n-1})\right)
            \,|\partial \mathbb{B}^n_1|^{\frac{1}{n-1}}\left(1+\frac{1}{n-1}\epsilon^{n-1}+O(\epsilon^{2n-2})\right) \\
      &= |\partial \mathbb{B}^n_1|^{\frac{1}{n-1}}
         \left(k +\frac{k}{n-1} \epsilon^{n-1} +O(\epsilon^n) \right) \\
      &> k|\partial \mathbb{B}^n_1|^{\frac{1}{n-1}} \\
      &= \sigma_k(\mathbb{B}^n_1) |\partial \mathbb{B}^n_1|^{\frac{1}{n-1}}
\end{align*}
where the inequality follows if $\epsilon$ is sufficiently small.      
\end{proof}

\section{Boundary connectedness in dimension at least 3} \label{section:surgery}

Let $(\Omega,g)$ be a compact, connected $n$-dimensional Riemannian manifold with boundary 
$\partial \Omega \neq \emptyset$, $n \geq 3$. The main theorem of this section shows
that $\Omega$ can be approximated by a connected subdomain with connected boundary so
that all three quantities $|\Omega|$, $|\partial\Omega|$, and $\sigma_1(\Omega)$
are changed by an arbitrarily small amount.

\begin{customthm}{\ref{theorem:connected-boundary}} 
Given $\epsilon>0$, there exists a domain $\Omega_\epsilon \subset \Omega$ with connected boundary and such that
\[
      |\Omega|-|\Omega_\epsilon|<\epsilon,\ ||\partial\Omega|-|\partial\Omega_\epsilon||<\epsilon,
      \ \mbox{and}\ |\sigma_1(\Omega)-\sigma_1(\Omega_\epsilon)|<\epsilon. 
\]
\end{customthm}
Since $\Omega$ has smooth boundary, we may extend $(\Omega,g)$ to a manifold $(M,g)$ so that
$\Omega$ is a domain compactly contained in $M$. Given points $p, \, q \in \partial \Omega$, let $\gamma: [0,l] \rightarrow M$ be a unit speed curve from $p$ to $q$ meeting $\partial \Omega$ orthogonally at $p$ and $q$. Consider Fermi coordinates $t, \, r, \, \theta^1, \ldots, \theta^{n-2}$ about $\gamma$, such that $t$ is the arclength parameter along $\gamma$, and $r, \, \theta_1, \ldots, \theta_{n-2}$ are geodesic normal coordinates on the slices $t=$constant. Assume that $\gamma$ extends beyond $p$ and $q$ so that 
\[
      \{ x \in M : d(x, \gamma) < \delta \} \cap \{ t = 0 \mbox{ or } l \} \cap \mbox{int} M = \emptyset
\]
for all $\delta \leq \delta_0$, for some fixed small $\delta_0 > 0$.
Let
\[
       T_\delta=\{ x \in \Omega : d(x, \gamma) = \delta \}
\]
and let 
\[
       \Omega_\delta=\Omega \setminus \{ x \in \Omega : d(x, \gamma) \leq \delta \}.
\]
\begin{proposition} \label{proposition:tube}
$\lim_{\delta \rightarrow 0} \sigma_1(\Omega_\delta)=\sigma_1(\Omega).$
\end{proposition}
The following lemma will be important later, since it implies that for a sequence of eigenfunctions, the $L^2$-norm on the boundary $\partial \Omega_\delta$ doesn't concentrate on the neck $T_\delta$ as $\delta \rightarrow 0$. 
\begin{lemma} \label{lemma:noconcentration}
If there are constants $\delta_0>0$, $C>0$ and a family of functions $u_\delta \in W^{1,2}(\Omega_\delta)$ with $\|u_\delta\|_{W^{1,2}(\Omega_\delta)} \leq C$ for $\delta \in (0,\delta_0)$, then
\[
      \lim_{\delta \rightarrow 0} \|u_\delta \|_{L^2(T_\delta)}=0.
\]
\end{lemma}
\begin{proof}
We may assume that the functions $u_\delta$ are defined on a neighborhood of the curve $\gamma$ on a larger domain $\tilde{\Omega}$ containing $\Omega$ and such that $\|u_\delta\|_{W^{1,2}(\tilde{\Omega}_\delta)} \leq C$.

We can also localize the support of $u_\delta$ to lie near the curve. Precisely, we choose
a number $r_0>0$ so that the the coordinates $(t,r,\theta)$ exist on the $r_0$ neighborhood of 
$\gamma$ and so that the metric is uniformly equivalent to the product metric $(0,l)\times D_{r_0}\setminus D_\delta$
given by $dt^2+dr^2+r^2g_{n-2}$ where $g_{n-2}$ denotes the standard metic on $S^{n-2}$
and $D_\sigma$ denotes the ball of radius $\sigma$ centered at the origin of $\mathbb R^{n-1}$. We choose
a cutoff function $\zeta(r)$ which is $1$ for $r\leq r_0/2$ and zero for $r\geq r_0$ and let
$v_\delta=\zeta u_\delta$. We then have by the Schwarz and arithmetic geometric mean inequalities
\[
|\nabla v_\delta|^2=u_\delta^2|\nabla \zeta|^2+2u_\delta \zeta\langle \nabla u_\delta,\nabla\zeta
\rangle+\zeta^2|\nabla u_\delta|^2\leq 2(\zeta^2|\nabla u_\delta|^2+u_\delta^2|\nabla\zeta|^2).
\]
This implies 
\[ 
\int_{\delta\leq r\leq r_0}|\nabla v_\delta|^2\leq c\int_{\Omega_\delta}(|\nabla u_\delta|^2+u_\delta^2)
\]
for a constant $c$ depending on $r_0$. Note that $r_0$ is fixed depending only on the geometry
and we will choose $\delta$ much smaller that $r_0$.

Thus to prove the lemma it suffices to show that for any $\epsilon>0$
\[
\int_{T_\delta}u_\delta^2=\int_{T_\delta}v_\delta^2\leq \epsilon \int_{\Omega_\delta}|\nabla v_\delta|^2
\]
for $\delta$ sufficiently small. Furthermore since the metric is uniformly equivalent to the
euclidean product metric on the support of $v_\delta$ it suffices the prove this estimate
for the product metric. This is what we shall do.

For a fixed $t_0\in(0,l)$ we consider the restriction which we denote by $v$, $v(r,\theta)=v_\delta(t_0,r,\theta)$ on the annulus $D_{r_0}\setminus D_\delta$ in $\mathbb R^{n-1}$. 
Choose $h$ to be the harmonic function of the annulus $D_{r_0}\setminus D_\delta$
\begin{align} \label{equation:harmonic}
         \Delta h &=0 
         \quad \mbox{ on } \, D_{r_0} \setminus D_{\delta} \no \\
         h &= v = 0  \quad \mbox{ on } \,  \partial D_{r_0} \\
         h &= v  \quad \mbox{ on } \,  \partial D_{\delta}. \no
\end{align}
By the Dirichlet minimizing property of $h$ we then have
\[
\int_{D_{r_0}\setminus D_\delta}|\nabla h|^2\leq \int_{D_{r_0}\setminus D_\delta}|\nabla v|^2.
\]
For any $\sigma$ with $\delta \leq \sigma \leq r_0$ we have by the divergence theorem
\[
     \int_{\partial D_{r_0}} \fder{h^2}{r} - \int_{\partial D_\sigma} \fder{h^2}{r}
     =\int_{D_{r_0} \setminus D_\sigma} \Delta h^2,
\]
and so from (\ref{equation:harmonic}) we get
\begin{align*}
     - \int_{\partial D_\sigma} \fder{h^2}{r} 
     &\leq 2 \int_{D_{r_0} \setminus D_\sigma} | \nabla h |^2  \\
     & \leq 2 \int_{D_{r_0} \setminus D_\sigma} | \nabla v |^2.  \\
  \end{align*}
Since we are working with respect to the standard metric on $\mathbb R^{n-1}$ the volume
measure on $\partial D_\sigma$ is $\sigma^{n-2}$ times that on the unit sphere 
$\partial D_1$. Therefore this may be rewritten
\[
       -\sigma^{n-2}\frac{d}{d\sigma} \left[ \sigma^{2-n}  \int_{\partial D_\sigma} h^2  \right]
       \leq 2 \int_{D_{r_0} \setminus D_\sigma} | \nabla v |^2\leq 2 \int_{D_{r_0} \setminus D_\delta} | \nabla v |^2
\]
since $\sigma\geq \delta$. Now we divide by $\sigma^{n-2}$ and integrate this with respect to $\sigma$
on the interval $[\delta,r_0]$ to obtain (note that $h=v$ on $\partial D_\delta$ and $h=0$ on
$\partial D_{r_0}$)
\[
       \delta^{2-n}\int_{\partial D_\delta}v^2\leq 2\left(\int_\delta^{r_0}\sigma^{2-n}d\sigma\right)\int_{D_{r_0} \setminus D_\delta} | \nabla v |^2.
 \]
This implies 
\[
\int_{\partial D_\delta}v^2\leq \epsilon_n(\delta)\int_{D_{r_0} \setminus D_\delta} | \nabla v |^2
\]
where $\epsilon_3(\delta)=2\delta\log(r_0/\delta)$ and $\epsilon_n(\delta)=\frac{2}{n-3}\delta$
for $n\geq 4$.

Written back in terms of $v_\delta$ this says that for each $t_0\in (0,l)$ we have
\[
\int_{\partial D_\delta}v_\delta(t_0,\delta,\theta)^2\leq \epsilon_n(\delta)\int_{D_{r_0} \setminus D_\delta} | \nabla^{n-1} v_\delta(t_0,r,\theta) |^2
\]
where we have used $\nabla^{n-1}$ to emphasize that the derivative is taken only along $t=t_0$.
We now integrate over $t_0\in (0,l)$ to obtain
\[
\int_{T_\delta}v_\delta^2\leq \epsilon_n(\delta)\int_0^l\int_{D_{r_0} \setminus D_\delta} | \nabla^{n-1} v_\delta(t_0,r,\theta) |^2\leq \epsilon_n(\delta)\int_{\Omega_\delta}|\nabla v_\delta|^2
\]
where we have used the inequality $|\nabla^{n-1}v_\delta|^2\leq |\nabla v_\delta|^2$.
Since $\epsilon_n(\delta)$ goes to $0$ as $\delta$ goes to $0$, we have completed the proof
with respect the euclidean product metric on $[0,l]\times (D_{r_0}\setminus D_\delta)$. As
discussed above this implies the result for the original metric and for any function $u_\delta$
in $W^{1,2}(\Omega_\delta)$. 
\end{proof}     

\begin{proof}[Proof of Proposition \ref{proposition:tube}]
Let $u_\delta$ be a first Steklov eigenfunction of $\Omega_\delta$ with eigenvalue $\sigma_1(\Omega_\delta)$, with $\|u_\delta\|_{L^2(\partial \Omega_\delta)}=1$. Then, 
\begin{align*}
       \begin{cases}
       \Delta u_\delta =0 & \mbox{ on } \Omega_\delta \\
       \fder{u_\delta}{\nu}=\sigma_1(\Omega_\delta) \, u _\delta & \mbox{ on } \partial \Omega_\delta.
       \end{cases}
\end{align*}

We first show that $\sigma_1(\Omega_\delta)$ is bounded from above independent of $\delta$
for $\delta$ small. To see this we use the variational characterization of $\sigma_1$
\[ 
    \sigma_1(\Omega_\delta)
  =\inf \{ \frac{\int_{\Omega_\delta}|\nabla f|^2}{\int_{\partial\Omega_\delta}f^2}:\ \int_{\partial\Omega_\delta }f=0\}
\]
where the infimum is taken over functions $f\in W^{1,2}(\Omega_\delta)$. Thus to get an upper 
bound we need only exhibit functions which integrate to $0$ over the boundary of $\Omega_\delta$
having bounded Rayleigh quotient. We can do this by choosing a {\it fixed} function which
is supported away from the tube region and so is a valid test function for any small $\delta$.

Elliptic boundary value estimates (\cite{M}) give bounds on $u_\delta$ and its derivatives up to $\partial \Omega_\delta$. There exists a sequence $u_{\delta_i}$ that converges in $C^2(K)$ on compact subsets $K \subset \Omega \setminus \gamma$ to a harmonic function $u$ on $\Omega \setminus \gamma$, satisfying 
\[
       \fder{u}{\nu}=\sigma u 
       \quad \mbox{ on } \quad \partial \Omega \setminus \{p,\, q\},
\]
with $\sigma =\lim_{i \rightarrow \infty} \sigma_1(\Omega_{\delta_i})$. Since $u_{\delta_i}$ converges to $u$ in $C^2(K)$ on compact subsets $K \subset \Omega \setminus \gamma$, there exists $C>0$ such that $\|u_{\delta_i}\|_{W^{1,2}(\Omega_{\delta_i})} \leq C$. By Lemma \ref{lemma:noconcentration}, $\lim_{i \rightarrow \infty} \|u_{\delta_i} \|_{L^2(T_{\delta_i})}=0$, and since $\|u_{\delta_i}\|_{L^2(\partial \Omega_{\delta_i})}=1$, $\|u\|_{L^2(\partial \Omega)} =1$. 

We now show that $u$ extends to a Steklov eigenfunction on $\Omega$. Consider the following logarithmic cut-off function about the curve $\gamma$,
\begin{equation} \label{equation:cut-off function}
 	\varphi_{\delta} =\left\{ \begin{array} {ll} 0 & r \le \delta^2  \\ 
	                   \frac{\log r - \log \delta^2 }{- \log \delta} & \delta^2 \le  r \le \delta \\ 
	                  1 & \delta \le r  \end{array} \right..
 \end{equation}
By the definition of $\varphi_\delta$, with respect to the product metric (see proof of Lemma \ref{lemma:noconcentration}) we have
\begin{align}
     \int_{\Omega}  |\nabla \varphi_\delta|^2
     &\leq \int_0^l\left(\int_{D_\delta \setminus D_{\delta^2}}  |\nabla \varphi_\delta|^2 \right) dt  \no \\
     &= \frac{C(n)}{(\log \delta)^2} \int_0^l 
                 \left(\int_{\delta^2}^{\delta} \frac{1}{r^2} \, r^{n-2}  \, dr \right) \, dt  \no \\
     &= \frac{C(n) \, l}{(\log \delta)^2}  \int_{\delta^2}^\delta r^{n-4} \, dr  \\
     &= \frac{C(n) \, l}{(\log \delta)^2} \, \cdot 
              \left\{ \begin{array}{ll}
           - \log \delta & \mbox{ if } n=3 \no \\
           & \\
           \frac{\delta^{n-3}(1-\delta^{n-3})}{n-3} & \mbox{ if } n > 3 
           \end{array} \right. \no \\           
     &= C(n) l \, \cdot 
              \left\{ \begin{array}{ll}
           - \frac{1}{\log \delta} & \mbox{ if } n=3 \\
           & \\
           \frac{1}{(\log \delta)^2}\frac{\delta^{n-3}(1-\delta^{n-3})}{n-3} & \mbox{ if } n > 3 
           \end{array} \right. \no \\  
     &\rightarrow 0 \mbox{ as } \delta \rightarrow 0. \label{equation:cut-off}
\end{align} 
Since the metric is uniformly equivalent to the product metric (see proof of Lemma \ref{lemma:noconcentration}), 
$\int_{\Omega}  |\nabla \varphi_\delta|^2 \rightarrow 0$ as $\delta \rightarrow 0$.
Let $\psi \in W^{1,2} \cap L^\infty (\Omega)$ and let $\psi_\delta=\varphi_\delta \psi$. Since $u$ is a Steklov eigenfunction with eigenvalue $\sigma$ on 
$\Omega \setminus \gamma$,
\begin{equation} \label{equation:weak-steklov}
    \int_{\Omega \setminus \gamma} \nabla u \nabla \psi_\delta
    = \sigma \int_{\Omega \setminus \gamma} u \psi_\delta.
\end{equation}
By (\ref{equation:cut-off}) and H\"older's inequality,
\[
     \int_{\Omega} \psi \nabla u \nabla \varphi_\delta \rightarrow 0
     \qquad \mbox{ as } \delta \rightarrow 0.
\]
Since $|\psi_\delta| \leq |\psi| \in L^\infty$ and $\psi_\delta \rightarrow \psi$ a.e., by the dominated convergence theorem, taking the limit of (\ref{equation:weak-steklov}) as $\delta \rightarrow 0$, we obtain
\[
    \int_{\Omega} \nabla u \nabla \psi
    = \sigma \int_{\partial \Omega} u \psi.      
\]
Therefore, $u$ extends to a Steklov eigenfunction with eigenvalue $\sigma$ on $\Omega$.

Finally, we show that $u$ is a {\em first} eigenfunction of $\Omega$; i.e. $\sigma=\sigma_1(\Omega)$. First, since $u_\delta$ is an eigenfunction corresponding to the first nonzero eigenvalue of $\Omega_\delta$, we have that $\int_{\partial \Omega_\delta} u_\delta=0$. Since $\lim_{\delta \rightarrow 0} \|u_\delta \|_{L^2(T_\delta)}=0$ (by Lemma \ref{lemma:noconcentration}), it follows that
\[
      \int_{\partial \Omega} u
      =\lim_{\delta \rightarrow 0} \int_{\partial \Omega_\delta} u_\delta
      =0.
\]
Therefore, $u$ is nonconstant, and $\sigma \geq \sigma_1(\Omega)$. Let $v$ be a first eigenfunction of $\Omega$ with $\|v\|_{L^2(\partial \Omega)}=1$. Let $\varphi_\delta$ be the logarithmic cut-off function defined by (\ref{equation:cut-off function}), and let
\[
       v_\delta =\varphi_\delta v 
            - \frac{1}{|\partial \Omega_{\delta^2}|} \int_{\partial \Omega_{\delta^2}} \varphi_\delta v.
\]
Then $\int_{\Omega_{\delta^2}} v_\delta=0$, and we will use $v_\delta$ as a test function in the variational characterization of the first nonzero Steklov eigenvalue $\sigma_1(\Omega_{\delta^2})$ of $\Omega_{\delta^2}$.
First note that since $\int_{\partial \Omega} v=0$,
\[
       \int_{\partial \Omega_{\delta^2}} \varphi_\delta v
       =\int_{\partial \Omega \cap \{ r < \delta\} } (\varphi_\delta -1) v.
\]
Then using H\"older's inequality and the fact that $\|v\|_{L^2(\partial \Omega)}=1$, we have
\begin{equation} \label{equation:average2}
     \left|\int_{\partial \Omega_{\delta^2}} \varphi_\delta v\right|
     \leq | \partial \Omega \cap \{ r < \delta\} |^{\frac{1}{2}}.
\end{equation}
Now,
\begin{align*}
     & \int_{\partial \Omega_{\delta^2}}  v_\delta^2 
       = \int_{\partial \Omega_{\delta^2}} (\varphi_\delta v)^2  
        - \frac{1}{|\partial \Omega_{\delta^2}|} \left(\int_{\partial \Omega_{\delta^2}} \varphi_\delta v \right)^2 \\
     &= \int_{\partial \Omega} v^2 
           -\int_{\partial \Omega \cap \{ r < \delta\} } (1-\varphi_\delta^2) v^2
       - \frac{1}{|\partial \Omega_{\delta^2}|} \left(\int_{\partial \Omega_{\delta^2}} \varphi_\delta v \right)^2 \\
    & \geq \int_{\partial \Omega} v^2 
        -C | \partial \Omega \cap \{ r < \delta\} |
        -\frac{| \partial \Omega \cap \{ r < \delta\} |}
                {|\partial \Omega| 
                  - |\partial \Omega \cap \{ r < \delta^2\} |} \\
     &= \int_{\partial \Omega} v^2 -C_1(\delta)        
\end{align*}
where $C_1(\delta) \rightarrow 0$ as $\delta \rightarrow 0$. For the inequality, the constant $C$ in the second term is a pointwise bound on the first eigenfunction $v$, and for the third term we used the estimate (\ref{equation:average2}).
On the other hand, 
\begin{align*}
    \int_{\Omega_{\delta^2}} |\nabla v_\delta |^2
    &= \int_{\Omega_{\delta^2}} |\nabla (\varphi_\delta v) |^2 \\
    &= \int_{\Omega_{\delta}} |\nabla v |^2 
        + \int_{\Omega_{\delta^2} \setminus \Omega_{\delta}} |\nabla (\varphi_\delta v) |^2 \\
    & \leq \int_{\Omega} |\nabla v|^2
        + 4\int_{\Omega_{\delta^2} \setminus \Omega_{\delta}} 
                 \left( |\nabla \varphi_\delta |^2 v^2+\varphi_\delta^2 |\nabla v |^2 \right) \\
    & \leq \int_{\Omega} |\nabla v|^2
         + C \int_{\Omega_{\delta^2} \setminus \Omega_{\delta}} |\nabla \varphi_\delta |^2
         + C |\Omega_{\delta^2} \setminus \Omega_\delta | \\
    & = \int_{\Omega} |\nabla v|^2 + C_2(\delta)
\end{align*}
with $C_2(\delta) \rightarrow 0$ as $\delta \rightarrow 0$, by (\ref{equation:cut-off}) and since $|\Omega_{\delta^2} \setminus \Omega_\delta | \leq |\{ \delta^2 < r < \delta, \, 0 < t < l \}|\rightarrow 0$ as $\delta \rightarrow 0$. Here, in the second inequality, the constant $C$ depends on a pointwise upper bound on $v$ and $|\nabla v|$. 
Combining these estimates, we have 
\[
    \sigma_1(\Omega_{\delta^2}) 
     \leq \frac{\int_{\Omega_{\delta^2}} |\nabla v_\delta |^2}{\int_{\partial \Omega_{\delta^2}}  v_\delta^2} 
     \leq \frac{\int_{\Omega} |\nabla v|^2 + C_2(\delta)}
                     {\int_{\partial \Omega} v^2 -C_1(\delta)}     
     \stackrel{\delta \rightarrow0}{\longrightarrow} \frac{\int_{\Omega} |\nabla v|^2}
                             {\int_{\partial \Omega} v^2} 
    = \sigma_1(\Omega).
\]
It follows that, 
$\sigma=\lim_{\delta \rightarrow 0} \sigma_1(\Omega_{\delta^2}) \leq \sigma_1(\Omega)$. Therefore,
\[
     \lim_{\delta \rightarrow 0} \sigma_1(\Omega_\delta)= \sigma_1(\Omega).
\]
\end{proof}

\begin{proof}[Proof of Theorem \ref{theorem:connected-boundary}]
Let $\Omega$ be a manifold with $b\geq 2$ boundary components. It suffices to construct
a sequence of connected smooth subdomains $\Omega_i$ with connected
boundary so that
\[
\lim_i|\Omega_i|=|\Omega|,\ \lim_i|\partial\Omega_i|=|\partial\Omega|,\ \mbox{and}\ 
\lim_i\sigma_1(\Omega_i)=\sigma_1(\Omega).
\]
To construct $\Omega_i$ we choose $b-1$ nonintersecting curves $\gamma_1, \ldots, \gamma_{b-1}$ which connect boundary components of $\Omega$ and meet $\partial \Omega$ orthogonally. Let $\Omega(\delta)$ be the domain with connected boundary obtained by removing a $\delta$-neighborhood of each of the curves from $\Omega$. Applying Proposition \ref{proposition:tube} finitely many times we obtain a sequence of domains $\Omega(\delta_j)$ with connected boundary, where 
$\delta_j \rightarrow 0$ as $j \rightarrow \infty$, such that
\[
     \lim_{j \rightarrow \infty} \sigma_1(\Omega(\delta_j)) = \sigma_1(\Omega).
\]
Since the $(n-1)$-dimensional volume of each tube $T_\delta$ tends to zero as $\delta \rightarrow 0$, 
\[
      \lim_{j \rightarrow \infty} |\partial \Omega(\delta_j)| \rightarrow |\partial \Omega|
\]
and so
\[
     \lim_{j \rightarrow \infty} \sigma_1(\Omega(\delta_j)) |\partial \Omega(\delta_j)|^{\frac{1}{n-1}}
     = \sigma_1(\Omega) |\partial \Omega |^{\frac{1}{n-1}}.
\]
It is clear that 
\[
\lim_j|\Omega(\delta_j)|=|\Omega|.
\]
Note that we can approximate the domains by smooth domains keeping the eigenvalue
and the volumes nearly constant. This completes the proof of Theorem \ref{theorem:connected-boundary}.
\end{proof}

We now apply Proposition \ref{proposition:tube} to show that the unit ball $\mathbb{B}^n_1$ in $\mathbb{R}^n$ does not maximize the first Steklov eigenvalue among contractible domains in $\mathbb{R}^n$.

\begin{theorem} \label{theorem:contractible}
There exists a family of bounded contractible smooth domains $\Omega_\delta \subset \mathbb{R}^n$, $0 < \delta \ll \epsilon < 1$, degenerating to $\mathbb{B}_1^n \setminus \mathbb{B}_\epsilon^n$ as $\delta \rightarrow 0$, such that
\[
       \lim_{\delta \rightarrow 0} \sigma_1(\Omega_\delta) 
       = \sigma_1(\mathbb{B}_1^n \setminus \mathbb{B}^n_\epsilon)\ \mbox{and}\ 
       \lim_{\delta \rightarrow 0} |\partial\Omega_\delta| 
       = |\partial(\mathbb{B}_1^n \setminus \mathbb{B}^n_\epsilon)|.
\]
\end{theorem}

\begin{proof}
Let $\gamma$ be the line segment $\{ \phi_1=0, \, \frac{\epsilon}{2} < \rho < 1 \}$, where $\phi_1$ denotes the angle with the positive $x_1$-coordinate axis in $\mathbb{R}^n$. Given $0<\delta \ll \epsilon$, let
\[
       \Omega_\delta
       =( \mathbb{B}^n_1\setminus \mathbb{B}^n_\epsilon ) \setminus 
           \{x \in \mathbb{B}^n_1\setminus \mathbb{B}^n_\epsilon \, : \, d(x,\gamma)< \delta \}.
\]
The result follows from Proposition \ref{proposition:tube}. Note that the domain so constructed
is only Lipschitz, but the corners can be smoothed while changing the boundary volume
and the first Steklov eigenvalue by an arbitrarily small amount.
\end{proof}

\begin{customthm}{\ref{theorem:ball}}
The unit ball $\mathbb{B}^n_1$ does not maximize the first Steklov eigenvalue among contractible domains in $\mathbb{R}^n$ having the same boundary volume. The maximum of $\sigma_1(\Omega)|\partial \Omega|^{\frac{1}{n-1}}$ among rotationally symmetric connected domains $\Omega \subset \mathbb{R}^n$ is achieved by $\mathbb{B}^n_1\setminus \mathbb{B}^n_\epsilon$ for some $0 < \epsilon < 1$.
\end{customthm}
\begin{proof}
For the contractible domain $\Omega_\delta$, $0< \delta \ll \epsilon$, defined in the proof of Theorem \ref{theorem:contractible}, $\lim_{\delta \rightarrow 0} |\partial \Omega_\delta| = |\partial (\mathbb{B}^n_1 \setminus \mathbb{B}^n_\epsilon)|$. Then, by Theorem \ref{theorem:contractible} and Proposition \ref{proposition:punctured ball}, 
\[
    \lim_{\delta \rightarrow 0} \sigma_1(\Omega_\delta) |\partial \Omega_\delta|
    =\sigma_1(\mathbb{B}_1^n \setminus \mathbb{B}^n_\epsilon) |\partial(\mathbb{B}_1^n \setminus \mathbb{B}^n_\epsilon)|
    > \sigma_1(\mathbb{B}^n_1) |\partial \mathbb{B}^n_1|.
\]
Therefore, for $\delta$ sufficiently small, 
$\sigma_1(\Omega_\delta) |\partial \Omega_\delta| > \sigma_1(\mathbb{B}^n_1) |\partial \mathbb{B}^n_1|$, and the unit ball $\mathbb{B}^n_1$ does not maximize the first Steklov eigenvalue among contractible domains in $\mathbb{R}^n$ having the same boundary volume.

The second statement follows, since a rotationally symmetric connected domain in $\mathbb{R}^n$ must be congruent to either $\mathbb{B}^n_1$ or $\mathbb{B}^n_1 \setminus \mathbb{B}^n_\epsilon$ for some $0 < \epsilon < 1$, and by Proposition \ref{proposition:punctured ball} $\sigma_1(\mathbb{B}_1^n \setminus \mathbb{B}^n_\epsilon) |\partial(\mathbb{B}_1^n \setminus \mathbb{B}^n_\epsilon)| > \sigma_1(\mathbb{B}^n_1) |\partial \mathbb{B}^n_1|$. Notice also that 
as $\epsilon$ tends to $1$ the eigenvalue $\sigma_1(\mathbb{B}_1\setminus \mathbb{B}_\epsilon)$ goes to $0$ (for
example, the coordinate functions have arbitrarily small Dirichlet integral and integrate to $0$
on the boundary), so the maximum is achieved for some $\epsilon$ between $0$ and $1$.
\end{proof}
We showed in Section 2 that the number
\[
\sigma^*(n)=\sup\{\sigma_1(\Omega)|\partial\Omega|^{\frac{1}{n-1}}:\ \Omega\subset\mathbb R^n\}
\]
is finite. We could similarly consider the number
\[
\sigma_0^*(n)=\sup\{\sigma_1(\Omega)|\partial\Omega|^{\frac{1}{n-1}}:\ \Omega\subset\mathbb R^n
\ \mbox{with}\ \partial\Omega\ \mbox{connected}\}. 
\]
\begin{corollary} 
We have $\sigma_0^*(2)<\sigma^*(2)$, but $\sigma_0^*(n)=\sigma^*(n)$
for $n\geq 3$.
\end{corollary}
\begin{proof}
From Weinstock's theorem we have $\sigma_0^*(2)=2\pi$, but we have $\sigma^*(2)>2\pi$
(cf . \cite[Proposition 4.2]{FS2} or \cite[Example 4.2.5]{GP4}).
On the other hand for $n\geq 3$, Theorem 4.1
shows that for any smooth domain $\Omega$, and any $\epsilon>0$ there is a domain 
$\Omega_0$ with connected boundary so that 
\[
\sigma_1(\Omega)|\partial\Omega|<\sigma_1(\Omega_0)|\partial\Omega_0|+\epsilon.
\]
It follows that $\sigma^*(n)\leq\sigma_0^*(n)$, and since the opposite inequality is
clear from the definition we have $\sigma^*(n)=\sigma_0^*(n)$.
\end{proof}
     
\bibliographystyle{plain}

\end{document}